\documentclass{amsart}

\newtheorem{thm}{Theorem}

\newtheorem{corollary}{Corollary}

\newtheorem{conjecture}{Conjecture}

\theoremstyle{definition}
\newtheorem{definition}{Definition}
\newtheorem{example}{Example}

\theoremstyle{remark}

\usepackage{tikz, outlines, enumerate, bm, amsmath, amsthm, amssymb, hyperref}

\def\area{\operatorname{area}}
\def\ldinv{\operatorname{ldinv}}
\def\pdinv{\operatorname{pdinv}}
\def\inv{\operatorname{inv}}

\title{A symmetric function lift of torus link homology}

\author{Andy Wilson}
\address
{Department of Mathematics \newline \indent
Kennesaw State University \newline \indent
Marietta, GA, 30060, USA}
\email{awils342@kennesaw.edu}


\begin{document}

\maketitle

\begin{abstract}
Suppose $M$ and $N$ are positive integers and let $k = \gcd(M, N)$, $m = M/k$, and $n=N/k$. We define a symmetric function $L_{M,N}$ as a weighted sum over certain tuples of lattice paths. We show that $L_{M,N}$ satisfies a generalization of Mellit and Hogancamp's recursion for the triply-graded Khovanov--Rozansky homology of the $M,N$-torus link. As a corollary, we obtain the triply-graded Khovanov--Rozansky homology of the $M,N$-torus link as a specialization of $L_{M,N}$. We conjecture that $L_{M,N}$ is equal (up to a constant) to the elliptic Hall algebra operator $\mathbf{Q}_{m,n}$ composed $k$ times and applied to 1. 
\end{abstract}

\section{Introduction}

For coprime positive integers $m$ and $n$, much has been discovered in recent years about the relationship between $m,n$-torus knots, the elliptic Hall algebra, and $m,n$-Dyck paths (lattice paths from $(0,0)$ to $(m,n)$ staying above the line $my = nx$). More precisely, the relevant objects are
\begin{enumerate}[(1)]
\item \label{kr} the triply-graded Khovanov--Rozansky homology of the $m,n$-torus knot,  
\item \label{sym} a certain symmetric function operator $\mathbf{Q}_{m,n}$ (defined in Section \ref{ssec:sym}) applied to 1, and
\item  \label{cat} a generating function over $m,n$-Dyck paths, weighted by variables $q$ and $t$ as well as monomials in variables $x_1, x_2, x_3 \ldots$.

\end{enumerate}
Gorsky and Negut conjectured that \eqref{sym} and \eqref{cat} are equal up to a sign \cite{gorsky-negut}. This conjecture was proved by Mellit \cite{mellit-rational}. An earlier result of Gorsky implies that \eqref{kr} appears as a certain specialization of \eqref{sym} and \eqref{cat} \cite{gorsky}. 

Somewhat less explored, but still fairly well understood, is the case where $m=n$, i.e.\ $m$ and $n$ are ``minimally coprime''. In this case, the objects are
\begin{enumerate}[(I)]
\item \label{kr_} the triply-graded Khovanov--Rozansky homology of the $n,n$-torus link (no longer a knot),
\item \label{sym_} the symmetric function $\nabla p_{1^n}$, where $\nabla$ is the Macdonald eigenoperator, and 
\item \label{cat_} a generating function over an arbitrary number of labeled boxes placed into $n$ columns \cite{torus-link-wilson}. 
\end{enumerate}
\eqref{sym_} and \eqref{cat_} are conjectured to be equal \cite{torus-link-wilson}. 
The same specialization as before allows one to move from \eqref{sym_} or \eqref{cat_} to \eqref{kr_}, making use of a recursion of Elias and Hogancamp \cite{elias-hogancamp}.

The goal of this work is to generalize both the $\gcd(m, n) = 1$ and $m=n$ cases to any positive integers $M$ and $N$. We let $k = \gcd(M, N)$, $m = M/k$, and $n=N/k$. The objects we consider are now
\begin{enumerate}[(A)]
\item \label{kr__} the triply-graded Khovanov--Rozansky homology of the $M,N$-torus link, 
\item \label{sym__} the elliptic Hall algebra operator $\mathbf{Q}_{m,n}$ applied iteratively $k$ times to 1, and
\item  \label{cat__} a generating function over $k$-tuples of (variations of) $m,n$-Dyck paths.
\end{enumerate}
The correct analog of Elias and Hogancamp's recursion for \eqref{kr__} is known due to Hogancamp and Mellit \cite{hogancamp-mellit}. 

The structure of this paper is as follows. In Section \ref{sec:background}, we give the necessary background on link homology and the elliptic Hall algebra to understand objects \eqref{kr__} and \eqref{sym__} above.  We develop the combinatorics of \eqref{cat__}, which is new work, in Section \ref{sec:combinatorics}. In Section \ref{sec:recursion}, we prove our main result, which is a recursion for our combinatorial generating functions. We close by stating our main conjecture in Section \ref{sec:conjecture}, a relationship between \eqref{sym__} and \eqref{cat__}, and explaining how the conjecture connects to previous work.

\section{Background}
\label{sec:background}

We suppose throughout the sequel that $M$ and $N$ are positive integers and let $k = \gcd(M, N)$, $m = M/k$, and $n=N/k$. In this section, we define the relevant notions from link homology theory and symmetric functions. 

\subsection{Torus link homology}
\label{ssec:links}

A \emph{link} is a subspace of $\mathbb{R}^3$ whose connected components are homeomorphic to circles. A connected link is called a  \emph{knot}. The goal of link invariants is to assign a mathematical object to each link (such as a number, a polynomial, or the homology of some chain complex) such that if two links are equivalent (usually by ambient isotopy) then the links are assigned the same object. We are particularly interested in Khovanov--Rozansky homology, which assigns a triply-graded homology to each link \cite{khovanov, bar-natan}. Khovanov--Rozansky homology generalizes many well-known link invariants such as the Alexander and Jones polynomials \cite{khovanov-icm}. Instead of computing Khovanov--Rozansky homology precisely, we will focus on computing the related trivariate generating function using variables $q$, $t$, and $a$.

We will focus on \emph{torus links}, a particular class of links. One way to depict a torus link is to begin by depicting the torus as $[0,1] \times [0,1]$ with vertical and horizontal edges identified, respectively. Given integers $M$ and $N$ with $\gcd(M, N) = k$, the $M,N$-torus link has components $My = Nx + \epsilon i$ for $i = 0$ to $k-1$ and some small value $\epsilon$, where each component wraps around the torus until it forms a knot. 

Hogancamp and Mellit derive a recursion for computing the triply graded Khovanov--Rozansky homology of any $M,N$-torus link \cite{hogancamp-mellit}. We describe their recursion below, following Gorsky, Mazin, and Vazirani's variation \cite{gorsky-mazin-vazirani}. Throughout, $v$ and $w$ are sequences in the alphabet $\{0,1,\bullet\}$. 

\begin{thm}[\cite{gorsky-mazin-vazirani, hogancamp-mellit}]
\label{thm:p_recursion}
For nonnegative integers $M$ and $N$, the triply graded Khovanov--Rozansky homology of the $M,N$-torus link is free over $\mathbb{Z}$ of graded rank $p(0^M, 0^N)$, which is an element of $\mathbb{N}[q,t^{\pm 1}, a, (1-q)^{-1}]$ computed by the following recursion:
\begin{enumerate}
\setcounter{enumi}{-1}
\item $p(\bullet^M, \bullet^N) = 1$.
\item $p(\bullet v, \bullet w) = p(v \bullet, w \bullet)$.
\item $p(0v, 0w) = (1-q)^{-1} p(v1, w1)$ if $|v| = |w| = 0$.
\item $p(0v, 0w) = t^{-\ell}p(v1, w1) + qt^{-\ell}p(v0, w0)$ if $\ell = |v| = |w| > 0$.
\item $p(1v, 0w) = p(v1, w\bullet)$.
\item $p(0v, 1w) = p(v\bullet, w1)$.
\item $p(1v, 1w) = (t^{|v|} + a) p(v\bullet, w\bullet)$.
\end{enumerate}
\end{thm}


Hogancamp and Mellit's original recursion can be obtained by removing all $\bullet$'s from $v$ and $w$. Our main result (Theorem \ref{thm:L_recursion}) is a lift of Theorem \ref{thm:p_recursion} to the level of symmetric functions. 

\subsection{Symmetric functions}
\label{ssec:sym}

We let $\Lambda$ denote the ring of symmetric functions in variables $x_1, x_2, x_3, \ldots$ over the ground field $\mathbb{Q}(q, t, a, z)$. We use traditional notation for the usual bases for $\Lambda$ such as the power sum, elementary, monomial, Schur, and Macdonald symmetric functions \cite{macdonald}.We will often employ \emph{plethystic substitution}, i.e.\ for any formal sum of Laurent monomials $A = a_1 + a_2 + a_3 + \ldots$ from the $x_i$'s or the ground field, we define
$$p_k[A] = a_1^k + a_2^k + a_3^k + \ldots$$
for any power sum polynomial $p_k$. We extend plethystic substitution to all of $\Lambda$ by viewing  the $p_i$'s as algebraically independent generators for $\Lambda$ \cite{plethysm}. We also let $$X = x_1 + x_2 + x_3 + \ldots.$$ Our first definition is an important operator on $\Lambda$ in the study of Macdonald polynomials. 

\begin{definition}
For any nonnegative integer $k$ and any symmetric function $f$, 
$$\mathbf{D}_k f[X] = \left. f[X + (1-q)(1-t)/z] \sum_{i \geq 0} (-z)^i e_i[X] \right|_{z^k} $$
where $\left. \right|_{z^k}$ extracts the coefficient of $z^k$ from the series to its left.
\end{definition}


\noindent
Next, we define the fundamental operators for the elliptic Hall algebra. 

\begin{definition}
For nonnegative integers $m$ and $n$ and $f \in \Lambda$, we define $\mathbf{Q}_{m, n} f$ by first setting
\begin{align*}
\mathbf{Q}_{0, n} f &= \frac{qt}{qt-1} h_n[(1-qt)X/(qt)] \cdot f \\
\mathbf{Q}_{1, n} f &= \mathbf{D}_n f .
\end{align*}
Otherwise, $m \geq 2$. We assume $m$ and $n$ are coprime, so there are unique integers $1 \leq a < m$ and $1 \leq b < n$ such that $na - mb = 1$. Then we let
$$\mathbf{Q}_{m, n} f = \frac{1}{(1-q)(1-t)} [ \mathbf{Q}_{m-a, n-b}, \mathbf{Q}_{a, b} ] f.$$
\end{definition}

$\mathbf{Q}_{m,n}$ is also defined when $m$ and $n$ are not coprime, but we will not need this level of generality \cite{rational-shuffle}. $\mathbf{Q}_{m,n}$ applied to 1 appears in the Rational Shuffle Theorem \cite{rational-shuffle}. Our next goal is to develop a conjectured combinatorial formula for $\mathbf{Q}_{m,n}^{k}(1).$ In order to state our formula, we will use the following operators
defined by Carlsson and Mellit in their proof of the Shuffle Theorem.

\begin{definition}
For any integer $\ell \geq 0$, we let
$$V_{\ell} = \mathbb{Q}[y_1, y_2, \ldots, y_{\ell}] \otimes \Lambda.$$
Following Carlsson and Mellit \cite{carlsson-mellit}, we define operators
\begin{align*}
d_{+} &: V_{\ell} \to V_{\ell + 1} \qquad (\ell \geq 0) \\
d_{-} &: V_{\ell} \to V_{\ell - 1} \qquad (\ell \geq 1)
\end{align*} by 
\begin{align*}
d_+ f &= T_1 T_2 \ldots T_{\ell} f[X + (t-1)y_{\ell+1}] \\
d_- f &= \left. -y_{\ell} f[X - (t-1)y_{\ell}] \sum_{i \geq 0} h_i[-X/y_{\ell}] \right|_{y_{\ell}^0}
\end{align*}
where
$$T_i f = \frac{(t-1) y_i f + (y_{i+1} - ty_i) s_i f}{y_{i+1} - y_i}$$
and $s_i$ swaps $y_i$ and $y_{i+1}$.
We define a third operator, 
$$d_{=} : V_{\ell} \to V_{\ell} \qquad (\ell \geq 1)$$
which also appears in Carlsson and Mellit's work but not by this name. It acts by 
$$d_{=} f = \frac{1}{t-1} \left(d_- d_+ f - d_+ d_- f \right).$$ 
\end{definition}

\noindent
Now we are ready to describe the relevant combinatorial objects.

\section{Combinatorial objects}
\label{sec:combinatorics}

Again, for positive integers $M$ and $N$ we let $k = \gcd(M, N)$, $m=M/k$, and $n=N/k$. We will define a generating function as a sum over certain $k$-tuples of lattice paths. These lattice paths will depend on a pair of sequences $v$ and $w$ of lengths $M$ and $N$, respectively, in the alphabet $\{0, 1, \bullet\}$. Furthermore, we let
$$|v| = \text{the number of 1's in $v$} $$
and insist that $|v| = |w|$. We will denote the resulting generating function by $L(v, w)$. 

\subsection{Path tuples}
\label{ssec:tuples}

\begin{definition}
An $m,n$-\emph{path}, or just a \emph{path},  is a sequence of $n$ unit-length north and $m$ unit-length east steps from $(a,0)$ to $(a+m, n)$ for some integer $a$ that
\begin{itemize}
\item begins with a north step, and
\item stays weakly above the line $my = nx.$
\end{itemize}
\end{definition}

We can imagine repeating the sequence of steps in a path to get an infinite path in the plane. Any height-$n$ ``band'' (region between $y = b$ and $y = b + n$ for some integer $b$) of the infinite path determines the original path. We will most often work with the $b=0$ band. Next, we define a labeling of the unit squares in each of $k$ ``sheets'' of $\left(\mathbb{Z}^2\right)^k$. 

\begin{definition}
We consider $\left(\mathbb{Z}^2\right)^k$ as $k$ \emph{sheets} of $\mathbb{Z}^2$, where the sheets are indexed from $0$ to $k-1$. 
The \emph{content} of a lattice square (or ``cell'') in the $i^{\text{th}}$ sheet, where $0 \leq i < k$, is
$$i + My - Nx$$
where $(x, y)$ are the coordinates of the lower right lattice point of the square.
\end{definition}

We depict the contents of some cells in Figure \ref{fig:paths}. Content provides a bijective correspondence between the cells in any band of $\left(\mathbb{Z}^2\right)^k$ and $\mathbb{Z}$. As a result, given a fixed band, we can refer to a cell by its content. In the sequel, we often use the phrase ``cell $c$'' to refer to the unique cell in the current band with content equal to $c$. 

\begin{definition}
For a fixed $k$-tuple of paths $\bm{P}$ in sheets $0, 1, \ldots, k-1$ of $\mathbb{Z}^2$, every cell $c \in \mathbb{Z}$ is either
\begin{itemize}
\item \emph{above $\bm{P}$}, i.e.\ the north step of $P$ in $c$'s row is to the right of $c$ and the east step of $P$ in $c$'s column is below $c$, or
\item \emph{below $\bm{P}$}.
\end{itemize}
Of the cells $c$ below $\bm{P}$, a cell $c$
\begin{itemize}
\item \emph{has a north step in $\bm{P}$} if there is a north step of $P$ on $c$'s left boundary and
\item \emph{has an east step in $\bm{P}$} if an east step of $P$ is on $c$'s upper boundary.
\end{itemize}
If $c$ is below $\bm{P}$ but it does not have a north step nor an east step in $\bm{P}$, then $c$ is \emph{strictly below $\bm{P}$}.
\end{definition}

Next, we use sequences $v$ and $w$ in the alphabet $\{0, 1, \bullet\}$ of lengths $M$ and $N$, respectively, such that 
$|v| = |w|$,
to restrict the set of paths we consider. These sequences $v$ and $w$ govern the relative location of the paths and the cells $0, 1, \ldots, M-1$  and $0, 1, \ldots, N-1$, respectively. 

\begin{figure}
\centering
\begin{tikzpicture}[scale=0.5]
\def\M{6}
\def\N{4}
\def\m{3}
\def\n{2}
\def\i{0}
\draw[gray, <->] (-3,0) -- (3,0);
\draw[gray, <->] (0,-1) -- (0,\n+1);
\draw[gray, dashed] (0,0) -- (\m, \n);
\draw[gray, dashed] (-3,2) -- (3,2);

\foreach \x in {-3,...,3} {
	\foreach \y in {0,1,...,\n} {
		\fill[color=gray] (\x,\y) circle (0.1);
        }
    }
\foreach \x in {-2,...,2} {
	\foreach \y in {1,...,\n} {
		\pgfmathsetmacro\c{int(\M*(\y-1) - \N*(\x)+\i)}
		\node at (\x-0.5, \y-0.5) {\small$\c$};
        }
    }
\draw[very thick, blue] (-2,1) -- (-1,1) -- (-1,2) -- (0,2);
\draw[very thick, blue, dashed] (-2,0) -- (-2,1);
\draw[very thick, blue, dashed] (1,2) -- (0,2);

\end{tikzpicture}
\hspace{5pt}
\begin{tikzpicture}[scale=0.5]
\def\M{6}
\def\N{4}
\def\m{3}
\def\n{2}
\def\i{1}
\draw[gray, <->] (-3,0) -- (3,0);
\draw[gray, <->] (0,-1) -- (0,\n+1);
\draw[gray, dashed] (0,0) -- (\m, \n);
\draw[gray, dashed] (-3,2) -- (3,2);

\foreach \x in {-3,...,3} {
	\foreach \y in {0,1,...,\n} {
		\fill[color=gray] (\x,\y) circle (0.1);
        }
    }
\foreach \x in {-2,...,2} {
	\foreach \y in {1,...,\n} {
		\pgfmathsetmacro\c{int(\M*(\y-1) - \N*(\x)+\i)}
		\node at (\x-0.5, \y-0.5) {\small$\c$};
        }
    }
\draw[very thick, blue] (-3,0) -- (-3,1) -- (-1,1);
\draw[very thick, blue, dashed] (-1,1) -- (0,1);
\draw[very thick, blue, dashed] (0,1) -- (0,2);
\end{tikzpicture}
\\
\begin{tikzpicture}[scale=0.5]
\def\M{6}
\def\N{4}
\def\m{3}
\def\n{2}
\def\i{0}
\draw[gray, <->] (-3,0) -- (3,0);
\draw[gray, <->] (0,-1) -- (0,\n+1);
\draw[gray, dashed] (0,0) -- (\m, \n);
\draw[gray, dashed] (-3,2) -- (3,2);

\foreach \x in {-3,...,3} {
	\foreach \y in {0,1,...,\n} {
		\fill[color=gray] (\x,\y) circle (0.1);
        }
    }
\foreach \x in {-2,...,2} {
	\foreach \y in {1,...,\n} {
		\pgfmathsetmacro\c{int(\M*(\y-1) - \N*(\x)+\i)}
		\node at (\x-0.5, \y-0.5) {\small$\c$};
        }
    }
\draw[very thick, red] (-2,1) -- (-2,2) -- (0,2);
\draw[very thick, red, dashed] (-2,0) -- (-2,1);
\draw[very thick, red, dashed] (1,2) -- (0,2);
\end{tikzpicture}
\hspace{5pt}
\begin{tikzpicture}[scale=0.5]
\def\M{6}
\def\N{4}
\def\m{3}
\def\n{2}
\def\i{1}
\draw[gray, <->] (-3,0) -- (3,0);
\draw[gray, <->] (0,-1) -- (0,\n+1);
\draw[gray, dashed] (0,0) -- (\m, \n);
\draw[gray, dashed] (-3,2) -- (3,2);

\foreach \x in {-3,...,3} {
	\foreach \y in {0,1,...,\n} {
		\fill[color=gray] (\x,\y) circle (0.1);
        }
    }
\foreach \x in {-2,...,2} {
	\foreach \y in {1,...,\n} {
		\pgfmathsetmacro\c{int(\M*(\y-1) - \N*(\x)+\i)}
		\node at (\x-0.5, \y-0.5) {\small$\c$};
        }
    }
\draw[very thick, red] (-3,0) -- (-3,1) -- (-1,1);
\draw[very thick, red, dashed] (-1,1) -- (0,1);
\draw[very thick, red, dashed] (0,1) -- (0,2);
\end{tikzpicture}
%
%
\caption{We have drawn the only two possible path tuples for  $v = 000110$ and $w=0110$, one in blue and one in red. The areas of the path tuples are 0 and 1, respectively. Both tuples have 4 path diagonal inversions.
}
\label{fig:paths}
\end{figure}
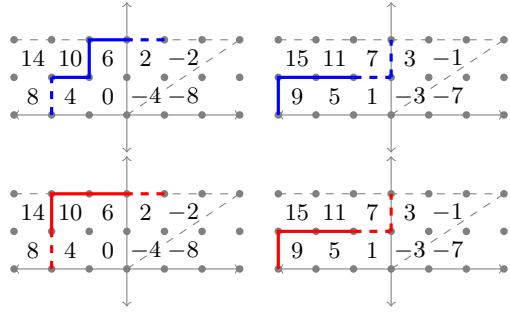

\begin{definition}
A $v,w$-\emph{path tuple} (or just a \emph{path tuple}) is a $k$-tuple of $m,n$-paths $\bm{P} = (P^{(0)}, P^{(1)}, \ldots, P^{(k-1)})$, one in each sheet, such that, for $v = v_0 \ldots v_{M-1}$,
\begin{enumerate}[(i)]
\item $v_i = 0$ if and only if cell $i$ is strictly below $\bm{P}$
\item \label{basement1} $v_i = 1$ if and only if cell $i$ is a north step of $\bm{P}$, and
\item $v_i = \bullet$ if and only if cell $i$ is above $\bm{P}$.
\end{enumerate}
and, for $w = w_0 \ldots w_{N-1}$, 
\begin{enumerate}[(i')]
\item $w_i = 0$ if and only if cell $i$ is strictly below $\bm{P}$
\item \label{basement2} $w_i = 1$ if and only if cell $i$ is an east step of $\bm{P}$, and
\item $w_i = \bullet$ if and only if cell $i$ is above $\bm{P}$.
\end{enumerate}
\end{definition}

North steps in cells $0 \leq c < M$ and east steps in cells $0 \leq c < N$ are called \emph{basement steps} and denoted with dashed lines in our figures.
As a sanity check, we note that, for any $M$ and $N$, the restriction that paths are weakly above the line $my = nx$ implies that there is exactly one path tuple if $v = \bullet^M$ and $w = \bullet^N$. On the other hand, if $v=0^M$ and $w=0^N$,
there are infinitely many $v,w$-path tuples. Figure \ref{fig:paths} contains two examples of path tuples.

\subsection{Invariant sets}
\label{ssec:invariant}

There is a class of objects known as \emph{$M,N$-invariant sets}, or just \emph{invariant sets}, that is in bijection with path tuples. These objects appear in work of Gorsky, Mazin, and Vazirani \cite{gorsky-mazin-vazirani}. 

\begin{definition}
An \emph{$M,N$-invariant set} is a set $\Delta \subset \mathbb{Z}_{\geq 0}$ with finite complement in $\mathbb{Z}_{\geq 0}$ such that, for every $i \in \Delta$, $i + M \in \Delta$ and $i + N \in \Delta$. We let $I_{M, N}$ denote the collection of all $M,N$-invariant sets.
\end{definition}

Given any $k$-tuple of $m,n$-paths $\bm{P}$, there is a natural partner invariant set $\Delta$ given by 
$$  i \in \Delta \Longleftrightarrow i \text{ is above } \bm{P}.$$
In fact, this is a bijection, since $\bm{P}$ can be recovered from $\Delta$. If we wish to consider $v,w$-path tuples $\bm{P}$, we get information about the intersections $\Delta \cap \{0,1,\ldots,M-1\}$ and $\Delta \cap \{0,1,\ldots,N-1\}$. 

\begin{definition}
An invariant set $\Delta$ \emph{fits} sequences $v$ and $w$ if its corresponding $k$-tuple of $m,n$-paths $\bm{P}$ is a $v,w$-path.
\end{definition}

We will occasionally use this identification with invariant sets in our proofs, although we will primarily work with path tuples. Next, we define statistics for path tuples.  

\subsection{Path tuple statistics}
\label{ssec:stats}

\begin{definition}
The \emph{area} of a path tuple $\bm{P}$ is the number of cells $c \geq M+N$ that are below $\bm{P}$. 
\end{definition}

\noindent
Next, we define a notion of diagonal inversion for a path.

\begin{definition}
A \emph{path diagonal inversion} in a path tuple $\bm{P}$ is a pair of cells $c < d$ with $c \geq 0$ and  $M \leq d < c + M$ such that 
\begin{itemize}
\item $c$ has a north step in $\bm{P}$ and
\item $d$ is (weakly) below $\bm{P}$. 
\end{itemize}
$c$ and $d$ may be in different sheets and the north step mentioned above can be a basement step. We let $\pdinv(\bm{P})$ denote the number of path diagonal inversions in a path tuple $\bm{P}$. 
\end{definition}

\subsection{The characteristic function of a path tuple}
\label{ssec:schroeder}

Given a path tuple $\bm{P}$, we describe how to obtain a sequence of $d_{+}$,  $d_{-}$, and $d_=$ operators that, when applied to 1, allow us to define $L(v,w)$. We also note that, in this subsection, we will break the $M,N$-symmetry that has existed thus far. 

\begin{definition}
A \emph{partial Schr\"oder path} is a lattice path from $(0,\ell)$ to $(N,N$) consisting of steps $(+1,0)$, $(0,+1)$ and $(+1,+1)$ that remains weakly above the line $y = x$ and does not contain any diagonal steps on the line $y=x$.
\end{definition}

\begin{definition}
\label{def:schroeder}
Given a path tuple $\bm{P}$, suppose $c_1 < c_2 < \ldots < c_N$ are the cells containing north steps in $\bm{P}$ and that $c_1, \ldots, c_{\ell}$ are the basement steps. Place the values $c_1, c_2, \ldots, c_N$ in the first $N$ cells on the line $y=x$ from bottom left to top right. Let $S$ be the (unique) partial Schr\"oder path from $(0,\ell)$ to $(N, N)$ such that 
\begin{enumerate}[(i)]
\item if $c_i < c_j < c_i + M$, then the entire (unique) square above $c_i$ and to the left of $c_j$ is below $S$,
\item if $c_i + M = c_j$, then the square above $c_i$ and to the left of $c_j$ contains a diagonal step,
\item otherwise, the square above $c_i$ and to the left of $c_j$ is above $S$.
\end{enumerate}
We call $S$ the \emph{partial Schr\"oder path} of $\bm{P}$. 
\end{definition} 

%
%
%
%
%
%

\begin{definition}
Given the partial Schr\"oder path $S$ of a path tuple $\bm{P}$, we begin with 1 and, reading $S$ from right to left, iteratively apply
\begin{itemize}
\item $d_+$ for any horizontal step,
\item $d_-$ for any vertical step, and 
\item $d_=$ for any diagonal step.
\end{itemize}
The resulting element of $V_{\ell}$ is the \emph{characteristic function} of $\bm{P}$, written $\chi(\bm{P})$. 
\end{definition}

The partial Schr\"oder paths and characteristic functions for the path tuples in Figure \ref{fig:paths} are depicted in Figure \ref{fig:schroeder}.
We close this subsection by defining our main objects of study.

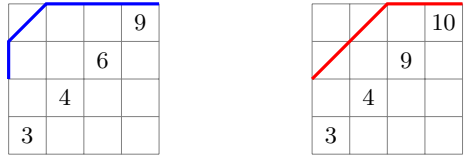
\begin{figure}
\begin{center}
\begin{tikzpicture}[scale=0.5]
\draw[gray, very thin] (0,0) grid (4,4);

\node at (0.5,0.5) {\small 3};
\node at (1.5,1.5) {\small 4};
\node at (2.5,2.5) {\small 6};
\node at (3.5,3.5) {\small 9};

\draw[very thick, blue] (0,2) -- (0,3) -- (1,4) -- (4,4);

\end{tikzpicture}
\hspace{50pt}
\begin{tikzpicture}[scale=0.5]
\draw[gray, very thin] (0,0) grid (4,4);

\node at (0.5,0.5) {\small 3};
\node at (1.5,1.5) {\small 4};
\node at (2.5,2.5) {\small 9};
\node at (3.5,3.5) {\small 10};

\draw[very thick, red] (0,2) -- (2,4) -- (4,4);

\end{tikzpicture}
\end{center}
\caption{This figure contains the partial Schr\"oder paths of the two path tuples from Figure \ref{fig:paths}. Their corresponding characteristic functions are $d_- d_= d_+ d_+ d_+(1)$ and $d_= d_= d_+ d_+ (1)$, both of which are in $V_2$.}
\label{fig:schroeder}
\end{figure}

\begin{definition}
Given sequences $v$ and $w$ of lengths $M$ and $N$, respectively, and $|v| = |w|$, we let
$$L(v, w)= \sum_{\bm{P}} t^{-\pdinv(\bm{P})} q^{\area(\bm{P})} \chi(\bm{P}) $$
where the sum is over all $v,w$-path tuples $\bm{P}$. 
\end{definition}

\begin{example}
To complete this subsection, we compute $L(v, w)$ in its entirety for $v = 000110$ and $w=0110$. There are two path tuples in this case, which appear in Figure \ref{fig:paths}. As mentioned in Figure \ref{fig:paths}, the areas of the path tuples are 0 and 1, respectively, and both tuples have 4 path diagonal inversions. Figure \ref{fig:schroeder} depicts the partial Schr\"oder paths for these path tuples. Assembling this information, and evaluating the corresponding sequences of Carlsson--Mellit operators, we get
\begin{align*}
L(000110, 0110) &= t^{-4} d_- d_= d_+ d_+ d_+(1) + t^{-4} q d_= d_= d_+ d_+ (1) \\ 
 &= -t^{-2} y_1 s_1 + t^{-3} q y_1 y_2 \in V_2.
\end{align*}
\end{example}

\subsection{Path labelings when $v=0^M$ and $w = 0^N$}

We can simplify the previous definitions for $L(v, w)$ when $v = 0^M$ and $w = 0^N$. As a result, we get an expression for $L(v,w)$ that does not explicitly use Carlsson--Mellit operators. We can assign labels to north steps in our paths in a manner reminiscent of parking functions. 

\begin{definition}
A \emph{labeled path tuple} is a $k$-tuple of paths $\bm{P}$ along with a function $f$ from the north steps of $\bm{P}$ to the positive integers such that, if $c + M = d$, then $f(c) < f(d)$. 
\end{definition}

We define a notion of diagonal inversions for labeled path tuples.

\begin{definition}
Given a labeled path tuple $(\bm{P}, f)$, a \emph{labeled diagonal inversion} is a pair of cells $c$ and $d$ with north steps such that $c < d < c + M$ and $f(c) < f(d)$. (Again, $c$ and $d$ may be in different sheets.) We let $\ldinv(\bm{P}, f)$ be the number of labeled diagonal inversions in the labeled path tuple $(\bm{P}, f)$.
\end{definition}

\begin{definition}
We let
$$L_{M,N} = \sum_{(\bm{P}, f)} t^{\ldinv(\bm{P}, f) - \pdinv(\bm{P})} q^{\area(\bm{P})}  \prod_{i > 0} x_i^{f^{-1}(i)} $$
where the sum is over all labeled path tuples $(\bm{P}, f)$.
 \end{definition}

\noindent
The following theorem comes directly from the work of Carlsson and Mellit \cite{carlsson-mellit}.

\begin{thm}
\label{thm:L_M_N}
For positive integers $M$ and $N$, 
$$L_{M,N} = L(0^M, 0^N).$$
\end{thm}

While it is theoretically possible to unwind the definitions of the Carlsson--Mellit operators to give a label-based expression for $L(v, w)$ for any $v$ and $w$, the resulting expression is quite technical and not helpful in what we aim to achieve here. Understanding this expression may be a worthwhile endeavor in the future.

\section{A recursion for $L(v,w)$}
\label{sec:recursion}

Now we are able to prove our main result, a recursion for $L(v,w)$, which contains the rank generating function of the Khovanov-Rozansky torus link homology of any torus link as a specialization.

\begin{thm}
\label{thm:L_recursion}
Let $v$ and $w$ be sequences in the alphabet $\{0,1,\bullet\}$. Assume in each case that both indexing sequences have the same number of 1's. We can compute $L(v, w)$ via the following recursion:
\begin{enumerate}
\setcounter{enumi}{-1}
\item \label{initial} $L(\bullet^M, \bullet^N) = 1$.
\item \label{bullets} $L(\bullet v, \bullet w) = L(v \bullet, w \bullet)$.
\item \label{all0} $L(0v, 0w) = (1-q)^{-1} d_- L(v1, w1)$ if $|v| = |w| = 0$.
\item \label{00} $L(0v, 0w) = t^{-\ell}d_{-} L(v1, w1) + qt^{-\ell}L(v0, w0)$ if $\ell = |v| = |w| > 0$.
\item \label{10}  $L(1v, 0w) = t^{-\ell} d_= L(v1, w\bullet)$ if $\ell = |v| = |w| -1$.
\item \label{01} $L(0v, 1w) = L(v\bullet, w1)$.
\item \label{11} $L(1v, 1w) = d_+ L(v\bullet, w\bullet)$
\end{enumerate}

\end{thm}

\begin{proof}
\eqref{initial} holds since there is exactly one such path -- the path corresponding to the invariant set $\mathbb{Z}_{\geq 0}$ -- and it does not have any area, path diagonal inversions, or north steps at least 0, so its characteristic function is 1. 

In \eqref{bullets}, cell 0 is above the path in $\bm{P}$ in sheet 0, so cells $M$ and $N$ must also be. Therefore cells $M-1$ and $N-1$ are above $\bm{P}'$ and the contribution of $\bm{P}'$ matches that of $\bm{P}$.

In \eqref{all0}, $0,1,\ldots,M-1$ and $0,1,\ldots,N-1$ are all strictly below $\bm{P}$. Therefore, every cell $c < M+N$ must be below $\bm{P}$. Let $r \geq M+N$ be the smallest cell above $\bm{P}$, so $r-N$ and $r-M$ are north and east steps of $\bm{P}$, respectively. Decrement every cell $r-M-N+1$ times to get a new path tuple $\bm{P}'$. $\bm{P}'$ has a north step in cell $(r-N)-(r-M-N+1) = M-1$ and an east step in cell $(r-M) - (r-M-N+1) = N-1$. Suppose $\bm{P}'$ had a north step in some cell $c < M-1$. Then $c+(r-M-N+1)$ is less than $(M-1) + (r-M-N+1) = r - N$, which has a north step in $\bm{P}$, so $c+(r-M-N+1)+N < r$ is above $\bm{P}$, which contradicts the choice of $r$. Similarly, suppose $\bm{P}'$ had an east step in some cell $d < N-1$. Since $d+(r-M-N+1)$ is less than $(N-1) + (r-M-N+1) = r-M$, which has an east step in $\bm{P}$, $d+(r-M-n+1)+M < r$ is above $\bm{P}$, which is another contradiction. Therefore $\bm{P}'$ does contribute to $L(0^{M-1}1, 0^{N-1}1)$, as desired. 

We still have to check that the $(1-q)^{-1} d_-$ factor is correct in \eqref{all0}. First, $M-1$ is a basement north step in $\bm{P}'$, whereas it was a non-basement north step in $\bm{P}$. Therefore a leading $d_-$ indeed differentiates their respective characteristic functions. Finally, we want to show that decrementing every cell by $r-M-N+1$ times lowers the area by $r-M-N$, since $r \geq M+N$. Every time we decrement, we know that $M+N$ is above the path, since if it is not we would have chosen a smaller $r$. After we decrement, the cell $M+N-1$ no longer contributes to the area statistic, since that statistic only counts cells $\geq M+N$. This argument holds until the last time we decrement, since in this case $M+N$ is already above the path before decrementing.

For the other statements, suppose $\bm{P}$ is a path tuple for sequences $iv$ and $jw$ for some symbols $i$ and $j$. $i$ and $j$ describe where the cell 0 is located with respect to the east and north steps of $\bm{P}$, respectively. Let $\bm{P}'$ be the path tuple that is obtained by replacing each content label $c$ in $\bm{P}$ with $c-1$. In other words, if $\Delta$ is the invariant set corresponding to $\bm{P}$, then $\bm{P}'$ is the path tuple whose invariant set $\Delta'$ is $\{i-1 : i \in \Delta\}$. We will show that $\bm{P}'$ is always a $vi', wj'$-path tuple for some symbols $i'$ and $j'$. Furthermore, since cells $M$ and $N$ appear directly above and to the left of cell 0, respectively, we can use $i$ and $j$ to relate the contribution of $\bm{P'}$ to $L(vi', wj')$ to the contribution of $\bm{P}$ to $L(iv, jw)$. 

In \eqref{00}, 0 is strictly below $\bm{P}$ but there is some cell $1 \leq c < M-1$ that is a north step of $\bm{P}$. The ``corner'' (east step followed by a north step) on the horizontal line on 0's north edge either appears at the upper left vertex of the 0 cell or to the left of this vertex. If it appears at the upper left vertex, then $M$ contains a north step and $N$ contains an east step. These steps become basement steps at $M-1$ and $N-1$, respectively, in $\bm{P}'$, and we get a $d_-$ in the first summand on the right-hand side of \eqref{00}. If the corner above 0 does not occur at the upper left vertex of cell 0, we do not gain any new basement steps, so $i=j=0$ and we do not see a new $d_-$ in this case. Furthermore, cell $M+N$ counts toward the area of $\bm{P}$ but not that of $\bm{P}'$. Finally, in either case, there are $\ell$ path diagonal inversions of the form $(c, M)$ in $\bm{P}$, one for each basement north step in $\bm{P}$, that are not path diagonal inversions in $\bm{P}'$, since the second coordinate is no longer at least $M$.

In \eqref{10}, 0 has a north step but not an east step. Therefore $M$ has a north step and $N$ is above $\bm{P}$. In $\bm{P}'$, $M-1$ has a north step and $N-1$ is above the path. Furthermore, there are $\ell$ path diagonal inversions of the form $(c,M)$ in $\bm{P}$ that are not path diagonal inversions in $\bm{P}'$. Finally, the partial Schr\"oder path of $\bm{P}$ has a diagonal step formed by 0 and $M$ that does not exist in the partial Schr\"oder path of $\bm{P}'$. 

In \eqref{01}, 0 has an east step but not a north step. Similarly, $M-1$ above $\bm{P}'$ and $N-1$ has an east step in $\bm{P}'$. Since $M+N$ is not below the path, the area is unchanged. Furthermore, since $M$ is not below $\bm{P}$, there are no path diagonal inversions of the form $(c, M)$, so this statistic is also unchanged by decrementing.

Finally, in \eqref{11}, 0 has a north step and an east step, which means that, in $\bm{P}'$, $M-1$ is to the left of a north step and $N-1$ is above an east step. 0 appears at the bottom of the partial Schr\"oder path for $\bm{P}$ but $-1$ cannot occur in the partial Schr\"oder path of $\bm{P}'$, which accounts for the $d_+$ in \eqref{11}.

\end{proof}

\begin{figure}
\begin{center}
\begin{tikzpicture}[scale=0.5]
\node at (0.5, 0.5) {\small$0$};
\node at (0.5, 1.5) {\small$M$};
\node at (-0.5, 0.5) {\small$N$};
\draw[very thick, dashed] (-1,0) -- (1,0) -- (1,2);
\node at (0.5,-2) {\eqref{bullets}};
\end{tikzpicture}
\hspace{20pt}
\begin{tikzpicture}[scale=0.5]
\node at (0.5,-2) {\eqref{all0}};
\node at (0.5, 0.5) {\small$0$};
\node at (0.5, 1.5) {\small$M$};
\node at (-0.5, 0.5) {\small$N$};
\draw[very thick, dashed] (-1,1) -- (0,1) -- (0,2);
\end{tikzpicture}
\hspace{20pt}
\begin{tikzpicture}[scale=0.5]
\node at (2,-2) {\eqref{00}};
\node at (0.5, 0.5) {\small$0$};
\node at (0.5, 1.5) {\small$M$};
\node at (-0.5, 0.5) {\small$N$};
\draw[very thick] (-1,1) -- (0,1) -- (0,2);
\node at (4.5, 0.5) {\small$0$};
\node at (4.5, 1.5) {\small$M$};
\node at (3.5, 0.5) {\small$N$};
\draw[very thick] (2,1) -- (3,1) -- (3,2);
\end{tikzpicture}
\begin{tikzpicture}[scale=0.5]
\node at (0.5,-2) {\eqref{10}};
\node at (0.5, 0.5) {\small$0$};
\node at (0.5, 1.5) {\small$M$};
\node at (-0.5, 0.5) {\small$N$};
\draw[very thick, dashed] (0,0) -- (0,1);
\draw[very thick] (0,1) -- (0,2);
\end{tikzpicture}
\hspace{20pt}
\begin{tikzpicture}[scale=0.5]
\node at (0.5,-2) {\eqref{01}};
\node at (0.5, 0.5) {\small$0$};
\node at (0.5, 1.5) {\small$M$};
\node at (-0.5, 0.5) {\small$N$};
\draw[very thick] (-1,1) -- (0,1);
\draw[very thick, dashed] (0,1) -- (1,1);
\end{tikzpicture}
\hspace{20pt}
\begin{tikzpicture}[scale=0.5]
\node at (0.5,-2) {\eqref{11}};
\node at (0.5, 0.5) {\small$0$};
\node at (0.5, 1.5) {\small$M$};
\node at (-0.5, 0.5) {\small$N$};
\draw[very thick, dashed] (0,0) -- (0,1) -- (1,1);
\end{tikzpicture}
\end{center}
\caption{We sketch the different possible configurations of steps near cells 0, $M$, and $N$ in a path tuple $\bm{P}$ which correspond to the different recursive cases in Theorem \ref{thm:L_recursion}.}
\label{fig:recursion_proof}
\end{figure}
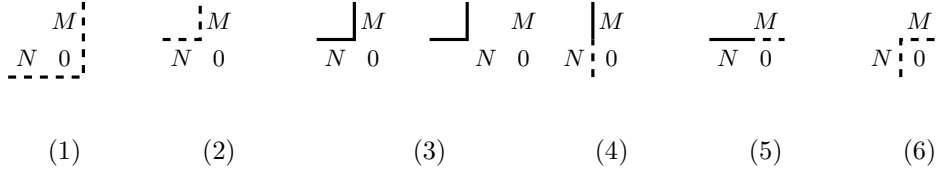

One can view Theorem \ref{thm:L_recursion} as a lift of Theorem \ref{thm:p_recursion} to the level of symmetric functions (or, more precisely, symmetric functions tensored with polynomials in variables $y_1, y_2, \ldots, y_{\ell}$). The result below shows how to recover any $p(v,w)$ from the corresponding $L(v,w)$. It is important to note that it is only at the level of $p(0^M, 0^N)$ that these objects are link invariants -- in particular, $L_{M,N}$ is \emph{not} a link invariant.

\begin{corollary}
Let $\psi$ be the operator\footnote{To experts in this area, applying $\psi$ equivalent to taking the ``Schr\"oder inner product.''} on $\Lambda$ defined by $\psi(e_i) = 1+a$ for each (algebraically independent) $e_i$. Then, for any sequences $v$ and $w$ in the alphabet $\{0,1,\bullet\}$ with $|v| = |w| = \ell$, 
$$ \psi\left( d_-^{\ell} L(v, w) \right) = p(v, w).$$
\end{corollary}

\begin{proof}
We work by induction in the order induced by the recursions in Theorems \ref{thm:p_recursion} and \ref{thm:L_recursion}. The recursive steps \eqref{initial} and \eqref{bullets} are directly equal. From \eqref{all0}, we have
\begin{align*}
\psi\left( L(0^M, 0^N) \right) &= \psi\left( (1-q)^{-1} d_- L(0^{M-1}1, 0^{n-1}1 \right) \\
&=  (1-q)^{-1} \psi\left( d_- L(0^{M-1}1, 0^{n-1}1 \right) \\
&= (1-q)^{-1} p(0^{M-1}1, 0^{N-1}) \\
&= p(0^M, 0^N)
\end{align*}
by induction and Theorems \ref{thm:p_recursion} and \ref{thm:L_recursion}. If $|v| = |w| = \ell > 0$, we compute
\begin{align*}
\psi\left( d_-^{\ell} L(0v, 0w) \right) &= \psi\left( t^{-\ell}d_{-}^{\ell+1} L(v1, w1) + qt^{-\ell}d_-^{\ell}L(v0, w0) \right) \\
&=  t^{-\ell} \psi\left(d_{-}^{\ell+1} L(v1, w1) \right) + qt^{-\ell} \left( d_-^{\ell}L(v0, w0) \right) \\
&= t^{-\ell} p(v1, w1) + qt^{-\ell} p(v0, w0) \\
&= p(0v, 0w). 
\end{align*}


The case in \eqref{10} is more interesting due to the presence of the $d_=$ operator. We consider
\begin{align*}
\psi \left( d_-^{\ell+1} L(1v, 0w) \right) &= \psi \left(  t^{-\ell} d_-^{\ell+1} d_= L(v1, w\bullet) \right)
\end{align*}
which we hope to show is equal to
\begin{align*}
p(1v, 0w) = p(v1, w\bullet) = \psi \left( d_{-}^{\ell+1} L(v1, w\bullet) \right).
\end{align*}
We will approach this combinatorially. We note that, for any Schr\"oder path $S$, 
$$\psi( \chi(S)) = \sum_{f} t^{\inv(f, S)} \prod f_i$$
where $f$ is any labeling of the cells on the line $y=x$ below $S$ with labels $a$ and 1 such that, for cell $c$ above $f_i$ and to the left of $f_j$, 
\begin{itemize}
\item if $c$ contains a diagonal step then $f_i = 1$,
\item if $c$ is below $S$ then $f_i$ and $f_j$ are not both $a$, and 
\item $c$ contributes to $\inv(f, S)$ exactly when $c$ is below $S$ and $f_i = 1$.
\end{itemize}
This follows from Haglund's work on shuffles and the fact that 
$$\psi(F) = \sum_{i=0}^{n} \langle F, h_i e_{n-i} \rangle a^i$$
for any symmetric function $F$ of degree $n$ \cite{haglund-book}. Let $S$ be any partial Schr\"oder path that contributes to $L(v1, w\bullet)$. Then the full Schr\"oder path $\texttt{n}^{\ell+1} S$ (where $\texttt{n}$ is a unit north step) contributes to $d_{\ell+1} L(v1, w\bullet)$ and $\texttt{n}^{\ell+1} \texttt{d} S$ contributes to $t^{-\ell} d_{-}^{\ell+1} d_= L(v1, w\bullet)$ (where $\texttt{d}$ is a diagonal step). The labelings $f$ for $\chi(\texttt{n}^{\ell+1} S)$ are in bijection with those for $\chi(\texttt{n}^{\ell+1} \texttt{d} S)$, where the bijection is simply prepending a 1. This 1 forms an inversion with exactly the next $\ell$ entries of $f$, which cancels out the $t^{-\ell}$ factor in $t^{-\ell} d_{-}^{\ell+1} d_= L(v1, w\bullet)$.

The case in \eqref{01} is straightforward. Finally, we need to understand the appearance of $d_+$ in \eqref{11}. In that case, we want to show that 
$$\psi \left( d_{-}^{\ell+1} L(1v, 1w) \right) = \psi \left( d_{-}^{\ell+1} d_+ L(v\bullet, w\bullet) \right)$$
is equal to 
$$p(1v, 1w) =  (t^{\ell} + a) p(v\bullet, w\bullet) = (t^{\ell} + a) d_{-}^{\ell} L(v\bullet, w\bullet)$$ 
where $\ell = |v| = |w|$. To get a labeling $f$ that contributes to the former equation, we simply prepend a 1 (yielding a factor of $t^{\ell}$) or an $a$ (yielding a factor of $a$) to get a labeling that contributes to the latter equation. 
\end{proof}

\section{A conjecture for $Q_{m,n}^{k}(1)$}
\label{sec:conjecture}

In this section, we state our main conjecture and explain how it generalizes previously known results.

\begin{conjecture}
\label{conj:main}
For positive integers $M$ and $N$, let $k = \gcd(M, N)$, $m = M/k$, $n=N/k$. Then
$$\mathbf{Q}_{m,n}^{k}(1) = \pm  (1-q)^k t^{C} L_{M,N}$$ 
where $C$ is the maximum of $\pdinv(\bm{P})$ over all $M,N$-path tuples $\bm{P}$. 
\end{conjecture}

\noindent
We currently do not know of a simple way to compute the value $C$ in Conjecture \ref{conj:main}, nor do we know a simple rule for the sign that appears.

\subsection{The Rational Shuffle Theorem}
\label{ssec:rational}

When $M$ and $N$ are coprime (so $M=m$, $N=n$, and $k=1$), the Rational Shuffle Theorem gives a combinatorial expression for $\mathbf{Q}_{m,n} (1)$.  We check that our formulas agree. Since $k=1$, there is only one sheet in each path tuple. We can decrement each path tuple until $M+N$ is the smallest value above the path, i.e.\ $M$ is the lowest-content north step and $N$ is the lowest-content east step. We obtain the usual depiction of rational parking functions by viewing such a path tuple in the horizontal band $1 < y \leq n+1$. 

\subsection{The $M=N$ case}
\label{ssec:M=N}

If $M=N=k$, then sheet $i$ is a single row containing all the cells equivalent to $i$ modulo $k$. We choose one cell as a north and east step in every sheet. This simplifies to the ``unbounded columns'' that appear in an open conjecture for $Q_{1,1}^{k} (1) = \nabla p_{1^k}$ \cite{torus-link-wilson}. 

\subsection{Conclusion}

A reasonable approach to proving Conjecture \ref{conj:main} would be to mimic this recursion to define and study extensions of the operators $\mathbf{Q}_{m,n}$ to sequences $v$ and $w$. A similar approach was used by Mellit to prove the Rational Shuffle Theorem \cite{mellit-rational}. An alternative approach would be to try to take advantage of the powerful ``Cauchy identity'' for non-symmetric Hall-Littlewood polynomials, following Blasiak, Haiman, Morse, Pun, and Seelinger \cite{blasiak-shuffle}. 

It would also be interesting to investigate whether $L(v, w)$ has a precise geometric or topological meaning in particular cases. For example, Hogancamp and Mellit show that $p(1^k 0^{m(k-1)}, 1^{k} 0^{n(k-1)})$ is related to colored homology of certain torus links \cite{hogancamp-mellit}.



\bibliographystyle{alpha}
\bibliography{statistics}

\end{document}